\documentclass[12pt]{amsart}

\usepackage{amsmath}
\usepackage{amssymb}
\usepackage{amsfonts}
\usepackage{amsthm}
\usepackage{enumerate}
\usepackage{color}
\usepackage{psfrag}
\usepackage{mathrsfs}
\usepackage{amscd}
\usepackage[hyperindex,backref=page]{hyperref}
\usepackage{graphicx,color}
\usepackage{tikz} 

\newtheorem{thm}{Theorem}[section]
\newtheorem{cor}[thm]{Corollary}
\newtheorem{lemma}[thm]{Lemma}
\newtheorem{prop}[thm]{Proposition}

\newtheorem{conj}[thm]{Conjecture}

\theoremstyle{definition}
\newtheorem{defi}[thm]{Definition}

\theoremstyle{remark}

\DeclareMathOperator{\link}{link}

\DeclareMathOperator{\dist}{dist}
\DeclareMathOperator{\st}{star}
\DeclareMathOperator{\depth}{depth}
\DeclareMathOperator{\pd}{pd}
\DeclareMathOperator{\height}{ht}

\DeclareMathOperator{\supp}{supp}

\DeclareMathOperator{\Min}{Min}

\DeclareMathOperator{\lcm}{lcm}

\DeclareMathOperator{\bou}{bou}
\DeclareMathOperator{\bc}{bc}

\begin{document}


\title[Symbolic depth functions of edge ideals via local graph structure]{Symbolic depth functions of edge ideals via local graph structure}

\author{Yuji Muta and Naoki Terai}

\address[Y. Muta]{Department of Mathematics, Okayama University, 3-1-1 Tsushima-naka, Kita-ku, Okayama 700-8530 Japan.}

\email{p8w80ole@s.okayama-u.ac.jp}

\address[N. Terai]{Department of Mathematics, Okayama University, 3-1-1 Tsushima-naka, Kita-ku, Okayama 700-8530, Japan.}
	
\email{terai@okayama-u.ac.jp}
\begin{abstract}
We study the depth of symbolic powers of edge ideals, focusing on local properties of graphs. First, we give several sufficient conditions for the symbolic depth function to be non-increasing, using local properties of graphs and local cohomology. Next, we give a counterexample to a conjecture of Minh, T.N. Trung, and Vu concerning the stable value of the symbolic depth function, where the difference between the conjectured value and the actual value can be arbitrary large. Finally, we study a conjecture of Muta stating the depth of a symbolic power of an edge ideal is at least that of the corresponding ordinary power. We prove it for graphs whose connected components are unicyclic or bipartite. 
\end{abstract}


\subjclass[2020]{Primary: 13F55, 13H10, 05C75; Secondary: 13D45, 05C90, 55U10}


\keywords{Edge ideals, depth, symbolic powers, ordinary power, local cohomology.}


\dedicatory{Dedicated to Professor Rafael H. Villarreal on the occasion of his 75th birthday}


\thanks{}
	

\maketitle

\section{Introduction}

The study of symbolic powers of a squarefree monomial ideal is an important topic in combinatorial commutative algebra. Let us recall the definition of symbolic powers. Let $K$ be a field, $S=K[x_{1},\ldots, x_{n}]$ be the polynomial ring in $n$ variables over $K$ and let $\mathfrak{m}$ be the unique homogeneous maximal ideal of $S$. Let $I$ be a squarefree monomial ideal of $S$. Then, for an integer $k\geq1$, the $k$-{\it th symbolic power }of $I$, denoted by $I^{(k)}$, is defined by 
$$I^{(k)}=\bigcap_{P \in \Min(I)} \big(I^k S_P \cap S\big).$$
If $\Min(I)=\{P_{1},\ldots, P_{r}\}$, then it is known that $I^{(k)}=P_{1}^{k}\cap\cdots\cap P_{r}^{k}$. 

The depth of symbolic powers of $I$ has been studied by many authors. It follows from a result of Hoa and T.N. Trung \cite[Theorem 4.7]{ht2010} that the function $k\mapsto\depth S/I^{(k)}$ is constant for all sufficiently large $k$. Varbaro \cite[Proposition 2.4]{v2011} proved that $$\min_{k\geq1}\depth S/I^{(k)}=n-\ell_{s}(I),$$
where $\ell_{s}(I)$ is the symbolic analytic spread of $I$, that is, $\ell_{s}(I)=\dim \bigoplus_{k}I^{(k)}/\mathfrak{m}I^{(k)}$. 
Later, Hoa, Kimura, Terai, and T.N. Trung proved that this minimum value is also the limit of the symbolic depth function. 

A natural problem is to determine when the symbolic depth function is non-increasing. Hoa, Kimura, Terai, and T.N. Trung \cite[Theorem 3.2]{hktt2017} proved that this property holds for the cover ideals, that is, the unmixed squarefree monomial ideals of height two. However, for general squarefree monomial ideals, the situation is different. In \cite[Theorem 2.8]{nt2019}, Nguyen and N.V. Trung showed that the symbolic depth function of squarefree monomial ideals need not be non-increasing. Later, the authors of this paper \cite[Theorem 4.17]{mt2025} gave an example of a pure shellable simplicial complex whose second symbolic powers is Cohen--Macaulay, but whose symbolic depth function is not non-increasing. More recently, Ficarra, Moradi, and Muta \cite[Proposition 5.5]{fmm2026Rees} provided an example of a complementary edge ideal whose symbolic depth function is not non-increasing. Here, complementary edge ideals are the squarefree monomial ideals generated by squarefree monomials of degree $n-2$. Moreover, in another paper \cite{fmm2026reg-depth}, they gave a more precise description of the symbolic depth function of complementary edge ideals.

On the other hand, it is still unknown whether the symbolic depth function of an edge ideal is always non-increasing. This problem was raised by Kimura, Terai, and Yassemi \cite[Problem 1.2]{kty2018} and was also stated by Nguyen and N.V. Trung \cite[Question 2.9]{nt2019}. 

Throughout, we assume that graphs always mean finite simple graphs. For a graph $G$ on the vertex set $[n]$, the {\it edge ideal }of $G$ is the squarefree monomial ideal generated by monomials corresponding to the edges, that is, $$I(G)=(x_{i}x_{j}\,\,:\{i,j\}\in E(G))\subset S.$$ 
Edge ideals of graphs were introduced by Villarreal in 1990 \cite{v1990} and have been extensively studied in combinatorial commutative algebra (see, for example, \cite{mv2012}). 

In this paper, we focus on the following conjecture. 

\begin{conj}[{\normalfont\cite[Problem 1.2]{kty2018} or \cite[Question 2.9]{nt2019}}]\label{conj of non-increasing}
Let $G$ be a graph. Then we have 
$$\depth S/I(G)^{(k)}\geq\depth S/I(G)^{(k+1)}\mbox{ for all }k\geq1.$$
\end{conj}

Several positive results are known. Kimura, Terai, and Yassemi \cite{kty2018} proved that 
Conjecture \ref{conj of non-increasing} holds if $G$ is very well-covered or has a leaf. In \cite[Theorem 3.5]{hlt2024}, Hien, Lam, and N.V. Trung studied the behavior of symbolic depth functions and proved that, for a connected graph, once the depth of a symbolic powers becomes one, it stays one for all higher symbolic powers. Minh, T.N. Trung, and Vu \cite{mtv2024} also studied the stable value of the symbolic depth function. 

The results for graphs with a leaf leads to the main problem of our paper. How much can the symbolic depth function be controlled by local information about a graph? In particular, can we find other local conditions that gives the non-increasing property, even when the graph has no leaf? More generally, can local information give inequalities between symbolic powers whose exponents differ by more than one?

We first recall a useful criterion based on minimal vertex covers. It is closely related to \cite[Remark 3.3]{sf2017}, where a similar colon identity was used to study Stanley depth. The same idea, together with a depth inequality for colon ideals, gives the following comparison between the depths of symbolic powers.

\begin{thm}[see, Theorem \ref{decreasing}]
Let $G$ be a  graph and $A\subset V(G)$, $t\in\mathbb{N}$.
If $|C\cap A|=t$ for any minimal vertex cover $C$ of $G$, then, we have 
$$\depth S/I(G)^{(k)}\geq\depth S/I(G)^{(k+t)}\mbox{ for all }k\geq1.$$
\end{thm}

Although this result gives a general condition, it is stated in terms of all minimal vertex covers of $G$. We therefore study how this condition can be obtained from the structure of an induced subgraph. In particular, we give sufficient conditions involving complete multipartite graphs and very well-covered graphs (see, Corollaries \ref{$t$-partite} and \ref{very well}). 

The following corollary extends the depth inequality for graphs with a leaf to graphs with a simplicial vertex.

\begin{cor}[see, Corollary \ref{simplicial}] 
Let $G$ be a graph. Suppose that $G$ has a simplicial vertex of degree $p$. Then we have
$$\depth S/I(G)^{(k)}\geq\depth S/I(G)^{(k+p)}\mbox{ for all }k\geq1$$
\end{cor}

We next investigate the non-increasing property from the viewpoint of local cohomology. We consider an edge $\{u,v\}$ of $G$ such that $N(u)\setminus\{v\}$ is a clique. Under this condition, we prove that, for every $k\geq1$ and $i\geq0$, $H_{\mathfrak m}^{i}(S/I(G)^{(k)})_{\bf a}\neq0$ implies that $H_{\mathfrak m}^{i}(S/I(G)^{(k+1)})\neq0$ provided that the multi-degree ${\bf a}$ satisfies certain conditions (see, Theorem \ref{non-vanishing}). As an application, we obtain the following sufficient condition for the non-increasing property as follows. 

\begin{cor}[see, Corollary \ref{deg2}]
Let $G$ be a graph. If there exists an edge $\{u,v\}\in E(G)$ with $\deg u=\deg v=2$, then we  have $$\depth S/I(G)^{(k)}\geq\depth S/I(G)^{(k+1)}\mbox{ for all }k\geq1.$$
\end{cor}

As a further application of Theorem \ref{non-vanishing}, in Corollary \ref{clique-corona}, we also obtain the non-increasing property for the multi-clique corona graphs introduced in \cite[Definition 5.1]{mt2026}. 

In the rest of Section \ref{sec3}, we discuss the following conjecture, proposed by Minh, T.N. Trung, and Vu in \cite{mtv2024}, relating the stable value of the symbolic depth function to the restricted bipartite connectivity number $\bc^{\prime}(G)$ of a graph $G$. 

\begin{conj}[\cite{mtv2024}, Conjecture 1.4]\label{conj for lim intro}
Let $G$ be a graph. Then we have 
$$\lim_{k\rightarrow\infty}\depth S/I(G)^{(k)}=\bc^{\prime}(G).$$
\end{conj}

The following theorem gives a counterexample to Conjecture \ref{conj for lim intro}. In these examples, $\bc^{\prime}(G)$ is larger than the stable value of the symbolic depth function, and the difference can be arbitrary large. 

\begin{thm}[see, Theorem \ref{counterexample}] 
Let $p\geq2$, and let $G=W(K_{p,p,p})$ be the whisker graph of the complete multipartite graph $K_{p,p,p}$. Then we have 
$$\lim_{k\rightarrow\infty}\depth S/I(G)^{(k)}\leq 2<p+1=\bc^{\prime}(G).$$
\end{thm}

In section \ref{sec4}, we discuss a comparison between depths of ordinary and symbolic powers of edge ideals. Simis, Vasconcelos, and Villarreal \cite{svv1994} proved that $I(G)^{(k)}=I(G)^{k}$ for all $k\geq1$ if and only if $G$ is bipartite. Hence, the depths of ordinary and symbolic powers are equal for bipartite graphs. In \cite{m2026}, Muta showed that, for every $k\geq2$, there exists a squarefree monomial ideal $I$ of height two such that $\depth S/I^{(k)}<\depth S/I^{k}$. These examples are given by cover ideals of certain graphs, and they are not edge ideals. Hence, the depth of a symbolic power can be smaller than that of the corresponding ordinary power for squarefree monomial ideal in general. However, no such example is known for edge ideals. Moreover, in the same paper, Muta proved that $\depth S/I(G)^{(k)}\geq\depth S/I(G)^{k}$ for every graph $G$ when $k=2,3$ and proposed the following conjecture. 

\begin{conj}[\cite{m2026}, Conjecture 1.6]
Let $G$ be a graph. Then we have 
$$\depth S/I(G)^{(k)}\geq\depth S/I(G)^{k}\mbox{ for all }k\geq1.$$
\end{conj}

We study this conjecture using local properties of graphs. Recall that a vertex $x$ of a graph $G$ is called {\it weakly simplicial }if the open neighborhood $N(x)$ of $x$ is an independent set and, for any $u,v\in N(x)$, either $N(u)\subset N(v)$ or $N(v)\subset N(u)$. We give a sufficient condition for the conjecture using a leaf whose neighbor is weakly simplicial. More precisely, we show that, when a graph has a leaf whose neighbor is weakly simplicial, the conjecture can be reduced to certain smaller induced subgraphs obtained by deleting vertices around the leaf (see, Theorem \ref{comparison}). The proof uses properties of colon ideals of ordinary and symbolic powers, together with an Eliahou--Kervaire splitting, introduced in \cite{ek1990}. 

As an application of this result, we prove the conjecture for graphs whose connected components are unicyclic. 

\begin{cor}[see, Corollary \ref{unicyclic}]
Let $G$ be a graph, Suppose that every connected component of $G$ is unicyclic or bipartite. Then we have 
$$\depth S/I(G)^{(k)}\geq\depth S/I(G)^{k}\mbox{ for all }k\geq1.$$
\end{cor}

\section{Preliminaries}

Throughout this paper, let $K$ be a field, let $S=K[x_{1},\ldots, x_{n}]$ be a polynomial ring over $K$, and let $\mathfrak{m}=(x_{1},\ldots, x_{n})$ be the unique homogeneous maximal ideal. We recall the notation and results used in this paper. We refer the reader to \cite{BHBook, HHBook, VBook} for the detailed information. 

\subsection{Graphs and simplicial complexes}
We assume that all graphs are finite and simple. For a graph $G$, we write $V(G)$ and $E(G)$ for its vertex set and edge set. For a subset $W\subset V(G)$, let $G|_{W}$ be the induced subgraph on $W$, and set $G-W=G|_{V(G)\setminus W}$. We write $G-v$ for $G-\{v\}$, and $G^{c}$ for the complement of $G$. For $v\in V(G)$, let $N_{G}(v)$ be its open neighborhood and set $N_{G}[v]=N(v)\cup\{v\}$. We omit the subscript when the graph is clear. A subset $C$ of $V(G)$ is called {\it vertex cover}, if $C\cap e\neq\emptyset$ for every edge $e$ of $G$. Also, if $C$ is a vertex cover that is minimal with respect to inclusion, then $C$ is called a {\it minimal vertex cover} of $G$. Moreover, a subset $A$ of $V(G)$ is called an {\it independent set}, if no two vertices in $A$ are adjacent to each other. In particular, if $A$ is an independent set of $G$ that is maximal with respect to inclusion, then $A$ is called a {\it maximal independent set} of $G$. A graph $G$ is {\it well-covered }if all its minimal vertex covers have the same cardinality. 

For a graph $G$ on the vertex set $[n]=\{1,\ldots, n\}$, the {\it edge ideal} of $G$, denoted by $I(G)$, is defined as 
$$I(G)=(x_{i}x_{j}\,\,:\{i,j\}\in E(G)).$$
For a subset $W\subset V(G)$, we write $K[W]$ for the polynomial ring in the variables corresponding to $W$. 

A \emph{simplicial complex} $\Delta$ on $[n]$ is a collection of subsets of $[n]$ closed under taking subsets, that is, $F\in\Delta$ and $H\subseteq F$ imply $H\in \Delta$. Elements of $\Delta$ are called {\it faces}, and maximal faces under inclusion are {\it facets}. We denote the set of facets of $\Delta$ by $\mathcal{F}(\Delta)$. For a face $F \in \Delta$, the {\it link} of $F$ in $\Delta$ is the simplicial complex $\link_{\Delta}F=\{H\subseteq[n]\setminus F\,\,:H\cup F\in\Delta\}$, and the {\it star} of $F$ is the simplicial complex $\st_{\Delta}F=\{H\in\Delta\,\,: H\cup F\in\Delta\}$. Also, for a subset $W\subset V(\Delta)$, set $\Delta_{W}=\{F\in\Delta\,\,: F\subset W\}$. We write $\langle F_{1},\ldots, F_{s}\rangle$ for the simplicial complex generated by $F_{1},\ldots, F_{s}$. 

\subsection{local cohomology modules and the depth}
For a monomial ideal $I$ of $S$, let $\mathcal{G}(I)$ denote the set of minimal monomial generators of $I$. For a nonzero finitely generated graded $S$-module $M$, we write $\pd M$ and $\depth M$ for its projective dimension and depth. For $i \geq 0$, we write $H^i_{\mathfrak{m}}(M)$ for the $i$-th {\it local cohomology module} of $M$ with respect to $\mathfrak{m}$. Then it is known that 
$$\depth M=\min\{i\,\,: H_{\mathfrak m}^{i}(M)\neq0\}.$$

\subsection{Symbolic powers}
Let $I$ be a squarefree monomial ideal of $S$. We denote by $\Min(I)$ the set of minimal prime ideals of $I$. For $k\geq1$, the $k$-th {\it symbolic power} of $I$ is defined by $$I^{(k)}=\bigcap_{P\in\Min(I)}(I^k S_P \cap S).$$
If $\Min(I)=\{P_{1},\ldots, P_{r}\}$, then it is known that $$I^{(k)}=P_{1}^{k}\cap\cdots\cap P_{r}^{k}.$$

For a vector $\mathbf{a} = (a_1,\dots,a_n) \in \mathbb{Z}^n$, we write ${\textbf x}^{\mathbf{a}} = x_1^{a_1}\cdots x_n^{a_n}$, $\supp({\bf a})=\{i\,\,:a_{i}\neq0\}$, and $\supp_{-}{\bf a}=\{i\,\,: a_{i}<0\}$. Let ${\bf e}_{i}$ be the $i$-th standard basis vector of $\mathbb{Z}^{n}$. For a subset $F \subseteq [n]$, let $S_F = S[x_j^{-1} : j \in F]$. The {\it degree complex }of a monomial ideal $I$ at ${\bf a}$ is 
$$\Delta_{\mathbf{a}}(I)=\{F\setminus\supp_{-}{\bf a}\,\,:\supp_{-}{\bf a} \subseteq F \subseteq [n], \ {\textbf x}^{\mathbf{a}} \notin I S_{F}\}.$$
Takayama's formula gives the following combinatorial description of the local cohomology module $H^i_{\mathfrak{m}}(S/I)_{\mathbf{a}}$ for a monomial ideal $I$.  
\begin{thm}[{\normalfont\cite[Theorem 1]{t2005}}]
\label{thm:Takayama}
Let $I\subset S$ be a monomial ideal. Then for every $\mathbf{a} \in \mathbb{Z}^n$ and every $i \geq 0$,
$$H^{i}_{\mathfrak{m}}(S/I)_{\bf a}\simeq\widetilde{H}^{i-|\supp_-({\mathbf{a}})|-1}(\Delta_{\bf a}(I);K).$$
\end{thm}

\section{the non-increasing property for symbolic powers}\label{sec3}

In this section, we study the non-increasing property of the depth of symbolic powers of edge ideals using local structures of graphs. We first give a condition on minimal vertex covers that yields inequalities between symbolic powers whose exponents differ by a fixed positive integer. We then study the non-vanishing of local cohomology and obtain further sufficient conditions for the symbolic depth function to be non-increasing. Finally, we give counterexamples to a conjecture on its stable value posed by Minh, T.N. Trung, and Vu. 

The following result follows from \cite[Remark 3.3]{sf2017} and \cite[Corollary 1.3]{r2010}. We include a proof for completeness. 

\begin{thm} \label{decreasing} Let $G$ be a  graph and $A\subset V(G)$, $t\in\mathbb{N}$.
If $|C\cap A|=t$ for any minimal vertex cover of $G$, then, we have 
$$\depth S/I(G)^{(k)}\geq\depth S/I(G)^{(k+t)}\mbox{ for all }k\geq1.$$
\end{thm}
\begin{proof}
Let $P_{C}$ be the minimal prime ideal of $I(G)$ corresponding to a minimal vertex cover $C$ of $G$. Set $m_A=\prod_{j\in A}x_j$ and fix $k\geq1$. We show that $P_{C}^{k+t}: m_{A}=P_{C}^{k}$.
Take a minimal monomial generator $m$ of $P_{C}^{k+t}$. Since $|C\cap A|=t$, we have   $(m):m_A \subset P_{C}^{k}$, and hence, we obtain $P_{C}^{k+t}:m_A\subset P_{C}^{k}$.
Conversely, take a minimal monomial generator $m$ of $P_{C}^{k}$. Since we have $|C\cap A|=t$, we have $mm_{A}\in P_{C}^{k+t}$, and hence, we see that $P_{C}^{k}\subset(P_{C}^{k+t}:m_A)$. Now we have
\begin{align*}
I(G)^{(k+t)}:m_{A}&=(\bigcap_{P \in \Min_{S}I(G)} P^{k+t}):m_{A} \\
&=\bigcap_{P\in\Min_{S}I(G)}(P^{k+t}:m_{A})\\   
&=\bigcap_{P \in \Min_{S}I(G)}P^{k}\\
&=I(G)^{(k)}.
\end{align*}
Note that  $m_{A}\not\in I(G)^{(k+t)}$, since $k\ge1$.
From \cite[Corollary 1.3]{r2010}, we have 
$$\depth S/I(G)^{(k+t)}\leq\depth S/(I(G)^{(k+t)}:m_{A})=\depth S/I(G)^{(k)},$$
which completes the proof. 
\end{proof}

We will apply Theorem \ref{decreasing} to induced subgraphs. The following definition describes vertices that have no neighbors outisde a given induced subgraph. 

\begin{defi}
Let $G$ be a graph and let $H$ be an induced subgraph of $G$. A vertex $v\in V(H)$ is $H$-{\it closed }in $G$ if $N_{G}(v)=N_{H}(v)$, or equivalently, if $\deg_{G}(v)=\deg_{H}(v)$. 
\end{defi}

Our first application uses a complete multipartite induced subgraph $H$. We choose one $H$-closed vertex from each part and apply Theorem \ref{decreasing} to the set of these vertices. 

\begin{cor} \label{$t$-partite} 
Let $t\geq2$ be an integer. Suppose a graph $G$ contains a complete $t$-partite graph $H$ as an induced subgraph. We assume each partition of $H$ has an $H$-closed vertex. Then we have 
$$\depth S/I(G)^{(k)}\geq\depth S/I(G)^{(k+t-1)}\mbox{ for all }k\geq1.$$
\end{cor}
\begin{proof}
We may assume that $G$ does not have a leaf. Choose an $H$-closed vertex for each partition of $H$. Let $A$ be the set  consisting of these vertices. By Theorem \ref{decreasing} it is enough to prove that $|C\cap A|=k-1$ for any minimal vertex cover $C$ of $G$. Since $G_{A}$ is a $t$-complete graph $K_{t}$, we see that $|C\cap A|\geq t-1$. Suppose $V(H)\subset C$.
Take an  an $H$-closed vertex $v$. Then $C\setminus \{v\}$ is a vertex cover of $G$.
This is a contradiction with the fact that $C$ is a minimal vertex cover of $G$. Hence there exists $w\in V(H)$ which does not belong to $C$. Then  all the vertices connecting to $w$  belong to $C$. Therefore, the $H$-closed vertices in the partition belonging to $w$  do not belong to $C$. Hence we have $|C\cap A|\geq t-1$. Then, we conclude that $|C\cap A|=t-1$. as required. 
\end{proof}

As another application of Theorem \ref{decreasing}, we consider very well-covered graphs. A graph $G$ is {\it very well-covered }graph if $G$ is well-covered, has no isolated vertices, and $\height I(G)=|V(G)|/2$. Every very well-covered graph has a perfect matching. Here, {\it perfect matching} is a set of pairwise disjoint edges covering all vertices. It is known that, by \cite{f1982}, every very well-covered graph has a perfect matching.

\begin{cor}\label{very well}
Let $G$ be a  graph which contains a very well-covered graph $H$ as an induced subgraph. Suppose that there exists an edge $\{u,v\}$ of a perfect matching of $H$ has both endpoints $H$-closed. Then we have 
$$\depth S/I(G)^{(k)}\geq\depth S/I(G)^{(k+1)}\mbox{ for all }k\geq1.$$
\end{cor}
\begin{proof}
Let $C$ be a minimal vertex cover of $G$ and let $M$ be a perfect matching of $H$ containing $\{u,v\}$. Take a minimal vertex cover $D$ of $H$ with $D\subset C\cap V(H)$. Since $H$ is a very well-covered, we have $|D|=|M|$. Notice that $|D\cap\{u,v\}|=1$. From now on, we verify that $C\cap\{u,v\}=D\cap\{u,v\}$. Let $x\in C\cap\{u,v\}$. Since $C$ is a minimal vertex cover of $G$, there exists a vertex $y\in N_{G}(x)\setminus C$. As $x$ is $H$-closed, we have $y\in V(H)$. Since $D$ is a vertex cover of $H$ and $y\notin D$, we obtain $x\in D$. Hence, we see that $C\cap\{u,v\}=D\cap\{u,v\}$, and thus, from Theorem \ref{decreasing}, the assertion follows. 
\end{proof}

The following corollary extends the depth inequality for graphs with a leaf to graphs with a simplicial vertex. 

\begin{cor} \label{simplicial} 
Let $G$ be a graph. Suppose that $G$ has a simplicial vertex of degree $p$. Then we have
$$\depth S/I(G)^{(k)}\geq\depth S/I(G)^{(k+p)}\mbox{ for all }k\geq1$$
\end{cor}
\begin{proof}
Let $v$ be a a simplicial vertex of degree $p$ in $G$. The assertion clearly holds if $p=0$, hence assume $p>0$. Let $C$ be a minimal vertex cover of $G$. Since $G|_{N[v]}$ is complete, we have $|C\cap N[v]|\geq p$. If $N[v]\subset C$, then $C\setminus\{v\}$ is a vertex cover of $G$, a contradiction. Hence, we obtain $|C\cap N[v]|=p$, and thus, from Theorem \ref{decreasing}, the assertion follows. 
\end{proof}

For further applications of Theorem \ref{decreasing}, we introduce the following notion. 

\begin{defi}
Let $G$ be a graph and let $M\subset V(G)$. For $T\subset V(G)\setminus M$, let $G_{T}$ be the graph obtained from $G$ by adding a new leaf adjacent to $v$ for each $v\in T$. We call $(G,M)$ {\it well-covered of cardinality} $c$ if $|C\cap V(G)|=c$  for every $T\subset V(G)\setminus M$ and minimal vertex cover $C$ of $G_{T}$. We say that $M$ is {\it minimal } if no proper subset of $M$ has this property.    
\end{defi}

The following corollary gives a depth inequality for graphs containing a well-covered marked graph as an induced subgraph. 

\begin{cor} \label{subgraph} 
Let $(H,M)$ be a graph $H$ with the marked vertex set $M$. Suppose $(H,M)$ is well-covered  of cardinality $c$. A graph $G$ contains $H$ as an induced subgraph such that any vertex $v\in M$ is $H$-closed in $G$. Then we have 
$$\depth S/I(G)^{(k)} \geq\depth S/I(G)^{(k+c)}\mbox{ for all }k\geq1$$
\end{cor}
\begin{proof}
Assume $c\geq1$. Let $C$ be a minimal vertex cover of $G$, and set $D=C\cap V(H)$, $T=D\setminus M$. Let $H_{T}$ be obtained from $H$ by adding a new leaf $z_{v}$ adjacent only to $v$ for each $v\in C\cap(V(H)\setminus M)$. Then, we see that $D$ is a vertex cover of $H_{T}$. We prove that $D$ is minimal. Suppose that there exists a vertex $v\in D$ such that $D\setminus\{v\}$ is a vertex cover of $H_{T}$. Since $\{u,z_{u}\}\in E(H_{T})$ for every $u\in T$, we have $v\notin T$, and hence, $v\in M$. Since $D\setminus\{v\}$ is a vertex cover of $H$, we obtain that $N_{G}(v)=N_{H}(v)\subset D\setminus\{v\}\subset C$, which implies that $C\setminus\{v\}$ is a vertex cover of $G$, a contradiction. Hence, $D$ is a minimal vertex cover of $H_{T}$, and thus, from assumption and Theorem \ref{decreasing}, the assertion follows. 
\end{proof}

We need the following lemma for the classification of well-covered marked graphs of cardinality three. 

\begin{lemma}\label{leaf-lemma}
Let $G$ be a connected graph with $|V(G)|\geq3$ and let $M\subset V(G)$. If there exists a leaf $u$ such that $u\in V(G)\setminus M$, then $(G,M)$ is not well-covered. 
\end{lemma}
\begin{proof}
Suppose that $(G,M)$ is well-covered of cardinality $c$ and $\{u,v\}\in E(G)$. Take $w\in N(v)\setminus\{u\}$ and let $A$ be a maximal independent set of $G$ such that $u,w\in A$. Set $C=V(G)\setminus A$. Since $V(G)\setminus A$ is a minimal vertex cover of $G$, we have $|C|=c$. Let $G^{\prime}$ be a graph obtained from $G$ by attaching a leaf $z$ to $u$. Since $C$ is a minimal vertex cover of $G$ and $u$ is a leaf, every vertex of $C\setminus\{v\}$ has a neighbor in $A\setminus\{u\}$, which implies that $C\cup\{u\}$ is a minimal vertex cover of $G^{\prime}$. Moreover, we have $|(C\cup\{u\})\cap V(G)|=c+1$, a contradiction, as required. 
\end{proof}

We now classify the connected well-covered marked graphs of cardinality three under the following assumptions. 

\begin{prop}\label{cardinality3} 
Let $(G,M)$ be a connected graph $G$, and let $M\subset V(G)$ such that $M$ does not contain a simplicial vertex of degree 1,2. Then $(G,M)$ is well-covered of  cardinality 3 such that $M$ is minimal if and only if, up to isomorphism of marked graphs, one of the following hold:
\begin{enumerate}
\item $G\cong C_{5}$ and $V(G)\setminus M$ consists of two non-adjacent vertices. 
\item $G\cong(P_{3}\sqcup K_{2})^{c}$ and $V(G)\setminus M$ consists of one end point of the $P_{3}$ in the complement of $G$. 
\item $V(G)=[5]$, $E(G)=\{\{1,2\},\{2,3\},\{3,4\},\{4,1\},\{1,5\},\{2,5\}\}$, and $M=\{1,2,3,4\}$. 
\item $G\cong K_{4}$ and $|M|=1$. 
\item $G\cong K_{3,3}$ and $M=V(G)$. 
\end{enumerate}
\end{prop}
\begin{proof}
If $G$ is unmixed with $\height I(G)=3$, then $4\leq|V(G)|\leq6$ from \cite[Corollary 3.4]{gv2005}. By the assumption on $M$ and Lemma \ref{leaf-lemma}, $G$ has no leaves. 
The list of all the connected unmixed graphs at most 6 vertices given in \cite{VBook}. By checking one by one we can obtain the above classification.
\end{proof}

We next study the non-increasing property from the viewpoint of local cohomology. We give conditions on the graph and the multidegree under which non-vanishing is preserved from one symbolic power to the next in the same cohomological degree. We then apply this result to obtain further sufficient conditions for the symbolic depth function to be non-increasing.

\begin{lemma}\label{negative-coord}
Let $G$ be a graph, and let $k\geq1$, $i\geq0$. Suppose that $H_{\mathfrak m}^{i}(S/I(G)^{(k)})_{\bf a}\neq0$ for some ${\bf a}\in\mathbb{Z}^{n}$. If $a_{v}<0$ and $\{v,w\}\in E(G)$, then we have $H_{\mathfrak m}^{i}(S/I(G)^{(k+1)})_{{\bf a}+{\bf e}_{w}}\neq0$.  
\end{lemma}
\begin{proof}
Since $H_{\mathfrak m}^{i}(S/I(G)^{(k)})_{\bf a}\neq0$ for some ${\bf a}\in\mathbb{Z}^{n}$, there exists a maximal independent set containing $\supp_{-}{\bf a}$, in particular, $\supp_{-}{\bf a}$ is independent. Since $v\in\supp_{-}{\bf a}$ and $\{v,w\}\in E(G)$, we know $w\notin\supp_{-}{\bf a}$. Fix a maximal independent set $A$ of $G$ with $\supp_{-}{\bf a}\subset A$. Since $v\in A$, we have $w\notin A$ and 
$$\sum_{j\notin A}(a_{j}+({\bf e}_{w})_{j})=\sum_{j\notin A}a_{j}+1.$$
Also, we have $\supp_{-}\{{\bf a}+{\bf e}_{w}\}=\supp_{-}{\bf a}$, and thus, we obtain $\Delta_{{\bf a}+{\bf e}_{w}}(I(G)^{(k+1)})=\Delta_{\bf a}(I(G)^{(k)})$. Therefore, we obtain that 
\begin{align*}
\dim_{K} H_{\mathfrak m}^{i}(S/I(G)^{(k+1)})_{{\bf a}+{\bf e}_{w}}&=\dim_{K}\widetilde{H}_{i-|\supp_{-}({\bf a}+{\bf e}_{w})|-1}(\Delta_{{\bf a}+{\bf e}_{w}}(I(G)^{(k+1)});K) \\
&=\dim_{K}\widetilde{H}_{i-|\supp_{-}{\bf a}|-1}(\Delta_{\bf a}(I(G)^{(k)});K) \\
&=\dim_{K} H_{\mathfrak m}^{i}(S/I(G)^{(k)})_{\bf a} \\
&>0,
\end{align*}
which completes the proof.

\end{proof}

We next give a local condition under which non-vanishing of local
cohomology is preserved from one symbolic power to the next in the
same cohomological degree.

\begin{thm}\label{non-vanishing}
Let $G$ be a graph on the vertex set $[n]$, $\{u,v\}\in E(G)$, and let $k\geq1$, $i\geq0$. Suppose that $N(u)\setminus\{v\}$ is a clique. Let ${\bf a}\in\mathbb{Z}^{n}$ such that $a_{u}\geq0,a_{v}\geq0$, and $0\leq a_{w}\leq a_{u}$ for each $N(u)\setminus\{v\}$. If $H_{\mathfrak m}^{i}(S/I(G)^{(k)})_{\bf a}\neq0$, then we have $H_{\mathfrak m}^{i}(S/I(G)^{(k+1)})\neq0$. 
\end{thm}
\begin{proof}
Fix $k\geq1$. Set $I=I(G)$ and $\Delta=\Delta_{\bf a}(I^{(k)})$. Suppose that $H_{\mathfrak m}^{i}(S/I(G)^{(k)})_{\bf a}\neq0$. From Takayama's formula \cite[Theorem 1]{t2005}, we know $\widetilde{H}_{i-|\supp_{-}{\bf a}|-1}(\Delta;K)\neq0$. Set ${\bf b}={\bf a}+{\bf e}_{u}+{\bf e}_{v}$. Notice that ${\supp}_{-}{\bf a}=\supp_{-}{\bf b}$. We claim two inclusions $$\Delta_{(V(\Delta)\setminus (N(u)\setminus\{v\}))}\subset \Delta_{\bf b}(I^{(k+1)})\subset\Delta,$$ where $\Delta_{(V(\Delta)\setminus (N(u)\setminus\{v\}))}$ denotes the induced subcomplex on the vertex set $V(\Delta)\setminus (N(u)\setminus\{v\})$. Take a face $F\in\Delta_{\bf b}(I^{(k+1)})$. Then there exists a maximal independent set $A$ such that ${\supp}_{-}{\bf a}\subset A$, $F\subset A\setminus\supp_{-}{\bf a}$, and $\sum_{j\notin A}b_{j}\leq k$. Since $\{u,v\}\in E(G)$, we have $|\{u,v\}\cap A|\leq1$, and thus, we obtain that 
$$\sum_{j\notin A}a_{j}=\sum_{j\notin A}b_{j}-|\{u,v\}\setminus A|\leq k-1,$$
which implies that $F\subset A\setminus\supp_{-}{\bf a}\in\Delta_{\bf a}(I^{(k)})$, and thus, we obtain the desired inclusion. Also, take a face $F\in\Delta|_{(V(\Delta)\setminus(N(u)\setminus\{v\})}$. Then there exists a maximal independent set $A$ of $G$ such that ${\supp}_{-}{\bf a}\subset A$, $F\subset A\setminus\supp_{-}{\bf a}$, and $\sum_{j\notin A}a_{j}\leq k-1$. We distinguish the following cases: 

{\bf Case 1.} Suppose that $A\cap\{u,v\}\neq\emptyset$. Since $\{u,v\}\in E(G)$, we know $|\{u,v\}\cap A|=1$, and hence, we obtain $\sum_{j\notin A}b_{j}=\sum_{j\notin A}a_{j}+1\leq k$. Hence, we see that $F\in\Delta_{\bf b}(I^{(k+1)})$. 

{\bf Case 2.} Suppose that $A\cap\{u,v\}=\emptyset$. Since $A$ is a maximal independent set of $G$ and $u\notin A$, we know $A\cap N(u)\neq\emptyset$. Also, since $v\notin A$, we have $A\cap (N(u)\setminus\{v\})\neq\emptyset$. On the other hand, since $(N(u)\setminus\{v\})$ is a clique and $A$ is independent, we obtain $|A\cap(N(u)\setminus\{v\})|\leq1$, which implies that there exists unique vertex $w\in(N(u)\setminus\{v\})$ such that $A\cap(N(u)\setminus\{v\})=\{w\}$. Set $B=(A\setminus\{w\})\cup\{u\}$. Fix $p\in A\setminus\{w\}$. If $\{u,p\}\in E(G)$, then we have $p\in N(u)$, a contradiction. Hence, $B$ is independent. Since $w\notin\supp_{-}{\bf a}$ and $\supp_{-}{\bf a}\subset A$, we obtain $\supp_{-}{\bf a}\subset B$. Take a maximal independent set $A^{\prime}$ such that $B\subset A^{\prime}$. Notice that $\supp_{-}{\bf a}\subset A^{\prime}$. Since $(A^{\prime}\setminus B)\cap\supp_{-}{\bf a}=\emptyset$, we have $a_{
j}\geq0$ for all $j\in A^{\prime}\setminus B$. Hence, it follows that 
$$\sum_{j\notin A^{\prime}}a_{j}\leq\sum_{j\notin B}a_{j}=\sum_{j\notin A}a_{j}+a_{w}-a_{u}\leq\sum_{j\notin A}a_{j}\leq k-1,$$
which implies that $$\sum_{j\notin A^{\prime}}b_{j}=\sum_{j\notin A^{\prime}}a_{j}+1\leq k.$$
Therefore, we see that $A^{\prime}\setminus\supp_{-}{\bf a}\in\Delta_{\bf b}(I^{(k+1)})$. Since $F\cap(N(u)\setminus\{v\})=\emptyset$ and $w\in N(u)\setminus\{v\}$, we have $w\notin F$, and thus, we obtain that $F\subset A\setminus(\supp_{-}{\bf a}\cup\{w\})\subset A^{\prime}\setminus\supp_{-}{\bf a}$. 

From now on, we prove that $H_{\mathfrak m}^{i}(S/I(G)^{(k+1)})\neq0$. We distinguish the following cases:

{\bf Case A.} Suppose that there exists $w\in(N(u)\setminus\{v\})\cap V(\Delta)$ such that $$\widetilde{H}_{i-|\supp_{-}{\bf a}|-2}(\link_{\Delta}w;K)\neq0.$$ Define a vector ${\bf c}\in\mathbb{Z}^{n}$ as $c_{j}=-1$ if $j=w$  and $c_{j}=a_{j}$ otherwise. Notice that $\supp_{-}{\bf c}=\supp_{-}{\bf a}\cup\{w\}$ and $\Delta_{\bf c}(I^{(k)})=\link_{\Delta}w$. Hence, from Takayama's formula, we obtain $$H_{\mathfrak m}^{i}(S/I^{(k)})_{\bf c}\neq0.$$ Since $c_{u}=a_{u}\geq0$, we have $\supp_{-}({\bf c}+{\bf e}_{u})=\supp_{-}{\bf c}$. Fix a maximal independent set $A$ of $G$ such that ${\supp}_{-}{\bf c}\subset A$. Since $c_{w}=-1$, we know $w\in\supp_{-}{\bf c}\subset A$. Also, since $w\in N(u)\setminus\{v\}$, we have $\{u,w\}\in E(G)$, and hence, we have $u\notin A$. Hence, we see that $\sum_{j\notin A}({\bf c}+{\bf e}_{u})_{j}=\sum_{j\notin A}c_{j}+1$, which implies that $\Delta_{\bf c}(I^{(k)})=\Delta_{{\bf c}+{\bf e}_{u}}(I^{(k+1)})$. Therefore, we obtain that 
\begin{align*}
\dim_{K} H_{\mathfrak m}^{i}(S/I^{(k+1)})_{{\bf c}+{\bf e}_{u}}&=\dim_{K}\widetilde{H}_{i-|\supp_{-}{\bf c}|-1}(\Delta_{{\bf c}+{\bf e}_{u}}(I^{(k+1)});K) \\
&=\dim_{K}\widetilde{H}_{i-|\supp_{-}{\bf c}|-1}(\Delta_{\bf c}(I^{(k)});K) \\
&=\dim_{K} H_{\mathfrak m}^{i}(S/I^{(k)})_{\bf c} \\
&>0,
\end{align*}
which completes the proof. 

{\bf Case B.} Suppose that $\widetilde{H}_{i-|\supp_{-}{\bf a}|-2}(\link_{\Delta_{\bf a}(I^{(k)})}w;K)=0$ for any $w\in(N(u)\setminus\{v\})\cap V(\Delta)$. We claim that the induced map 
$$\iota_{*}:\widetilde{H}_{i-|\supp_{-}{\bf a}|-1}(\Delta_{V(\Delta)\setminus(N(u)\setminus\{v\})};K)\rightarrow\widetilde{H}_{i-|\supp_{-}{\bf a}|-1}(\Delta;K)$$
from the inclusion map $\iota:\Delta_{V(\Delta)\setminus(N(u)\setminus\{v\})}\rightarrow\Delta$ is surjective. 
We prove by induction on $t=|(N(u)\setminus\{v\})\cap V(\Delta)|$. If $t=0$, then we have $\Delta_{V(\Delta)\setminus(N(u)\setminus\{v\})}=\Delta$, and hence, the assertion trivially holds. Hence, assume that $t>0$ and take $w\in N(u)\setminus\{v\})\cap V(\Delta)$. Now, we have 
$$\Delta=\Delta_{V(\Delta)\setminus\{w\}}\cup\st_{\Delta}w\mbox{ and }\Delta_{V(\Delta)\setminus\{w\}}\cap\st_{\Delta}w=\link_{\Delta}w.$$
Hence, since $\st_{\Delta}w$ is a cone, from the Mayer--Vietoris exact sequence and the assumption, we obtain that 
$$\widetilde{H}_{i-|\supp_{-}{\bf a}|-1}(\Delta_{V(\Delta)\setminus\{w\}};K)\rightarrow\widetilde{H}_{i-|\supp_{-}{\bf a}|-1}(\Delta;K)$$
is surjective. Fix $z\in(N(u)\setminus\{v,w\})\cap V(\Delta)$ and $F\in\link_{\Delta}z$. If $w\in F$, then, we know that $\{w,z\}\subset F\cup\{z\}$ and $w,z$ contain a clique $N(u)\setminus\{v\}$, a contradiction. Hence, we obtain $w\notin F$, and thus, we see that $F\cup\{z\}\in\Delta_{V(\Delta)\setminus\{w\}}$. Therefore, we get $\link_{\Delta_{V(\Delta)\setminus\{w\}}}z=\link_{\Delta}z$, ane thus, we obtain that $$\widetilde{H}_{i-|\supp_{-}{\bf a}|-2}(\link_{\Delta_{V(\Delta)\setminus\{w\}}}z;K)\simeq\widetilde{H}_{i-|\supp_{-}{\bf a}|-2}(\link_{\Delta}z;K)=0\mbox{ for each }z\in N(u)\setminus\{v,w\}.$$
Then, by the induction hypothesis, we see that the inclusion map $\Delta_{V(\Delta)\setminus(N(u)\setminus\{v\})}\rightarrow\Delta_{V(\Delta)\setminus\{w\}}$ induces a surjective 
$$\widetilde{H}_{i-|\supp_{-}{\bf a}|-1}(\Delta_{V(\Delta)\setminus(N(u)\setminus\{v\})};K)\rightarrow\widetilde{H}_{i-|\supp_{-}{\bf a}|-1}(\Delta_{V(\Delta)\setminus\{w\}};K).$$
Therefore, the composition of these maps on reduced homology is exactly the map induced by the inclusion $\Delta_{V(\Delta)\setminus(N(u)\setminus\{v\})}\rightarrow\Delta$. Now, we have 
$$\widetilde{H}_{i-|\supp_{-}{\bf a}|-1}(\Delta_{V(\Delta)\setminus(N(u)\setminus\{v\})};K)\rightarrow\widetilde{H}_{i-|\supp_{-}{\bf a}|-1}(\Delta_{\bf b}(I^{(k+1)});K)\rightarrow\widetilde{H}_{i-|\supp_{-}{\bf a}|-1}(\Delta;K).$$
By the claim, we know that this composition is surjective. Since $\widetilde{H}_{i-|\supp_{-}{\bf a}|-1}(\Delta;K)\neq0$, we obtain that $\widetilde{H}_{i-|\supp_{-}{\bf a}|-1}(\Delta_{\bf b}(I^{(k+1)});K)\neq0$, which completes the proof. 
\end{proof}

Combining Lemma \ref{negative-coord} with Theorem \ref{non-vanishing}, we obtain the following graph-theoretic condition for the symbolic depth function to be non-increasing.

\begin{cor}\label{clique}
Let $G$ be a graph. If there exists $Q\subset V(G)$ with $|Q|\geq2$ such that $G|_{Q}$ is complete and $|N(u)\setminus Q|\leq1$ for every $u\in Q$. Then we have 
$$\depth S/I(G)^{(k)}\geq\depth S/I(G)^{(k+1)}\mbox{ for all }k\geq1.$$
\end{cor}
\begin{proof}
Fix $k\geq1$ and set $i=\depth S/I(G)^{(k)}$. Take ${\bf a}\in\mathbb{Z}^{n}$ such that $H_{\mathfrak m}^{i}(S/I(G)^{(k)})_{\bf a}\neq0$. If $a_{z}<0$ for some $z\in Q$, then we take $w\in Q\setminus\{z\}$. Since $\{z,w\}\in E(G)$, from Lemma \ref{negative-coord}, we obtain $H_{\mathfrak m}^{i}(S/I(G)^{(k+1)})\neq0$. Hence, we assume that $a_{z}\geq0$ for every $z\in Q$. Take $u\in Q$ such that $a_{u}=\max\{a_{z}\,\,: z\in Q\}$. If $N(u)\setminus Q\neq\emptyset$, let $v$ be its unique vertex. Otherwise, take $v\in Q\setminus\{u\}$. In either case, we have $\{u,v\}\in E(G)$ and $N(u)\setminus\{v\}\subset Q\setminus\{u\}$. If $a_{v}<0$, then, from Lemma \ref{negative-coord}, we obtain $H_{\mathfrak m}^{i}(S/I(G)^{(k+1)})\neq0$. Hence, we assume that $a_{v}\geq0$. Since $G|_{Q}$ is complete and $N(u)\setminus\{v\}\subset Q\setminus\{u\}$, we have $N(u)\setminus\{v\}$ is a clique. Moreover, we see that $0\leq a_{w}\leq a_{u}$ for every $w\in N(u)\setminus\{v\}$. Therefore, from Theorem \ref{non-vanishing}, we obtain $H_{\mathfrak m}^{i}(S/I(G)^{(k+1)})\neq0$, which completes the proof. 
\end{proof}

Taking $Q=\{u,v\}$ in Corollary \ref{clique}, we obtain the following consequence. 

\begin{cor}\label{deg2}
Let $G$ be a graph. If there exists an edge $\{u,v\}\in E(G)$ with $\deg u=\deg v=2$, then we  have $$\depth S/I(G)^{(k)}\geq\depth S/I(G)^{(k+1)}\mbox{ for all }k\geq1.$$
\end{cor}

For the first two marked graphs in Proposition \ref{cardinality3}, we obtain the non-increasing property, improving the depth inequality in Corollary \ref{subgraph}.

\begin{cor} \label{non-increasing} 
Let $G$ be a graph. Suppose $G$ is a graph which contains one of the following $H$ as an induced subgraph:
\begin{enumerate}
\item $H\cong C_{5}$ and there exists $M\subset V(H)$ such that $V(H)\setminus M$ consists of two non-adjacent vertices and every vertex of $M$ is $H$-closed in $G$
\item $H\cong(P_{3}\sqcup K_{2})^{c}$ and there exists $M\subset V(H)$ such that $V(H)\setminus M$ consists of an endpoints of the $P_{3}$ in $H^{c}$, and every vertex of $M$ is $H$-closed in $G$. 
\end{enumerate}
Then we have 
$$\depth S/I(G)^{(k)} \geq\depth S/I(G)^{(k+1)}\mbox{ for all }k\geq1.$$
\end{cor}
\begin{proof}
(1) Now, we see that $H|_{M}$ contains an edge $\{u,v\}$. Since both $u$ and $v$ are $H$-closed, we have $\deg_{G}u=\deg_{G}v=2$, and thus, from Corollary \ref{deg2}, the assertion follows. 

(2) Write $V(H)=\{a,b,c,d,e\}$ and $E(H^{c})=\{\{a,b\},\{b,c\},\{d,e\}\}$, where $a$ is the unique vertex $a\in V(H)\setminus M$. From Theorem \ref{decreasing}, it is verify that $|C\cap\{b,d\}|=1$ for every minimal vertex $C$ of $G$. Fix a minimal vertex cover $C$ of $G$. Suppose $b,d\in C$. Then, by the minimality of $C$ and $N_{G}(b)=\{d,e\}$, we know $e\notin C$, and hence, we get $a,c\in C$. Hence, we obtain that $N_{G}(d)=\{a,b,c\}\subset C$, which is a contradiction. Therefore, from Theorem \ref{decreasing}, the assertion follows. 
\end{proof}

The authors of this paper \cite[Definition 5.1]{mt2026} introduced multi-clique corona graphs. We obtain the following result on the depth of their symbolic powers. Let us recall the definition of multi-clique corona graphs. 
For a graph $G$ on the vertex set $X_{[h]}$ and positive integers $n_{1},\ldots, n_{h}$, we set 
$$\mathcal{H}=\bigcup_{1\leq i\leq h}\{H_{i}^{1},\ldots, H_{i}^{n_{i}}\},$$
where $H_{i}^{j}$ is a nonempty graph indexed by the vertex $x_{i}$. We define the {\it multi-corona graph} $G\circ\mathcal{H}$ of $G$ and $\mathcal{H}$ is the disjoint union of $G$ and $H_{i}^{1},\ldots, H_{i}^{n_{i}}$, with additional edges joining each vertex $x_{i}$ to all vertices $H_{i}^{j}$. 
We call $G\circ\mathcal{H}$ is {\it multi-clique corona graph}, if $H_{i}^{j}$ is the complete graph for all $i$ and $j$. The next result follows from Corollaries \ref{simplicial} and \ref{clique}.

\begin{cor}\label{clique-corona}
Let $G$ be a graph on the vertex set $X_{[h]}$ and positive integers $n_{1},\ldots, n_{h}$ and let $\mathcal{H}=\bigcup_{1\leq i\leq h}\{H_{i}^{1},\ldots, H_{i}^{n_{i}}\}$. If $G\circ\mathcal{H}$ is a multi-clique corona graph, then we have
$$\depth S/I(G\circ\mathcal{H})^{(k)}\geq\depth S/I(G\circ\mathcal{H})^{(k+1)}\mbox{ for all }k\geq1.$$
\end{cor}
\begin{proof}
Take an attached clique $H_{i}^{j}$ of $G$. If $|V(H_{i}^{j})=1|$, then the assertion follows from Corollary \ref{simplicial}. Suppose that $|V(H_{i}^{j})|\geq2$, and set $Q=V(H_{i}^{j})$. Then, $G\circ\mathcal{H}|_{Q}$ is complete, and we have $N_{G\circ\mathcal{H}}(v)=(Q\setminus\{v\})\cup\{x_{i}\}$ for all $v\in Q$. Therefore, from Corollary \ref{clique}, the assertion follows.   
\end{proof}

In the rest of this section, we discuss a conjecture posed by Minh, T.N. Trung, and Vu \cite[Conjecture 1.4]{mtv2024} concerning the stable value of the depth of symbolic powers of edge ideals. We first recall the necessary definitions. 

Let $G$ be a graph, and let $\mathcal{B}(G)$ be the set of maximal induced bipartite subgraphs of $G$. For $H\in\mathcal{B}(G)$, write $H=H_{1}\sqcup\cdots\sqcup H_{c}\sqcup \{p_{1},\ldots p_{t}\}$, where $H_{1},\ldots, H_{c}$ are the connected components containing an edge and $\{p_{1},\ldots p_{t}\}$ is the set of isolated vertices. $p_{1},\ldots, p_{t}$ are called {\it clustered} if there exists $u\in V(G)\setminus V(H)$ such that the subgraph induced by $\{u,p_{1},\ldots, p_{t}\}$ is a bouquet, where a bouquet is a star graph. Let $\bou_{G}(H)$ be the least number of bouquets needed to cluster all isolated vertices of $H$ and $0$ if $t=0$. The number of restricted connected components of $H$ is $c^{\prime}(H)=c+\bou_{G}(H)$. The restricted bipartite connectivity number of $G$ is 
$$\bc^{\prime}(G)=\min\{c^{\prime}(H)\,\,: H\in\mathcal{B}(G)\}.$$
Minh, T.N. Trung, and Vu proposed the following conjecture. 

\begin{conj}[\cite{mtv2024}, Conjecture 1.4]\label{conj for lim}
Let $G$ be a graph. Then we have 
$$\lim_{k\rightarrow\infty}\depth S/I(G)^{(k)}=\bc^{\prime}(G).$$
\end{conj}

The following example provide a counterexample to Conjecture \ref{conj for lim} and shows that the difference between the restricted bipartite connectivity number and the stable value of the symbolic depth function can be arbitrary large. 

\begin{thm}\label{counterexample}
Let $p\geq2$, and let $G=W(K_{p,p,p})$ be the whisker graph of the complete multipartite graph $K_{p,p,p}$. Then we have 
$$\lim_{k\rightarrow\infty}\depth S/I(G)^{(k)}\leq 2<p+1=\bc^{\prime}(G).$$
\end{thm}
\begin{proof}
Write $V(K_{p,p,p})=X_{1}\sqcup X_{2}\sqcup X_{3}$ and $X_{i}=\{x_{i,1},\ldots, x_{i,p}\}$ for $1\leq i\leq 3$. For each $i,j$, let $y_{i,j}$ be the leaf adjacent only to $x_{i,j}$, and put $Y_{i}=\{y_{i,1},\ldots, y_{i,p}\}$. 

First, we verify that $\bc^{\prime}(G)=p+1$. Let $H\in\mathcal{B}(G)$. By maximality, we have $V(H)=X_{i}\cup X_{j}\cup Y_{1}\cup Y_{2}\cup Y_{3}$ for some distinct $i,j\in\{1,2,3\}$. Write $\{\ell\}=\{1,2,3\}\setminus\{i,j\}$. Then $H$ has exactly one connected component containing an edge, and its isolated vertices are precisely the $p$ vertices of $Y_{\ell}$. Since each vertex of $V(G)\setminus V(H)=X_{\ell}$ is adjacent to exactly one vertex of $Y_{\ell}$, we have $\bou_{G}(H)=p$. Thus, $c^{\prime}(H)=p+1$, and hence $\bc^{\prime}(G)=p+1$. 

Next, we prove that $\lim_{k\rightarrow\infty}\depth S/I(G)^{(k)}\leq 2$. Define a vector ${\bf a}\in\mathbb{N}^{|V(G)|}$ as $a_{x_{i,j}}=1$ and $a_{y_{i,j}}=0$ for every $i$ and $j$. The maximal independent set of $G$ are the form $A\cup\{y_{i,j}\,\,:x_{i,j}\notin A\}$, where $A$ ranges over the independent sets of $K_{p,p,p}$. Since $\sum_{u\notin A\cup\{y_{i,j}\,\,:x_{i,j}\notin A\}}a_{u}=3p-|A|$ and $|A|\leq p$, we see that $\sum_{u\notin A\cup\{y_{i,j}\,\,:x_{i,j}\notin A\}}a_{u}\leq 2p$ if and only if $A=X_{1},X_{2}$, or $X_{3}$. Hence, we obtain $$\Delta_{\bf a}(I(G)^{(2p+1)})=\langle X_{1}\cup\bigcup_{j\neq1}Y_{j}\rangle\cup\langle X_{2}\cup\bigcup_{j\neq2}Y_{j}\rangle\cup\langle X_{3}\cup\bigcup_{j\neq3}Y_{j}\rangle.$$
Set $\Gamma_{i}=\langle X_{i}\cup\bigcup_{j\neq i}Y_{j}\rangle$ for $1\leq i\leq 3$. Then, Since $\Gamma_{1}\cap\Gamma_{2}=\langle Y_{3}\rangle$, we know that $\Gamma_{1}\cup\Gamma_{2}$ is connected, which implies that $\widetilde{H}_{0}(\Gamma_{1}\cup\Gamma_{2};K)=0$. Also, we see that $(\Gamma_{1}\cup\Gamma_{2})\cap\Gamma_{3}=\langle Y_{1}\rangle\sqcup\langle Y_{2}\rangle$ is disconnected. Since $\Gamma_{3}$ is connected, from the Mayer--Vietoris exact sequence, we see that $$\widetilde{H}_{1}(\Delta_{\bf a}(I(G)^{(2p+1)});K)\rightarrow\widetilde{H}_{0}((\Gamma_{1}\cup\Gamma_{2})\cap\Gamma_{3};K)$$
is surjective. Since $\widetilde{H}_{0}((\Gamma_{1}\cup\Gamma_{2})\cap\Gamma_{3};K)\neq0$, we obtain $\widetilde{H}_{1}(\Delta_{\bf a}(I(G)^{(2p+1)});K)\neq0$. From Takayama's formula, we get $H_{\mathfrak m}^{2}(S/I(G)^{(2p+1)})_{\bf a}\neq0$, and thus, from \cite[Theorem 2.4]{hktt2017}, we obtain that $$\lim_{k\rightarrow\infty}\depth S/I(G)^{(k)}=\min_{k\geq1}\depth S/I(G)^{(k)}\leq2,$$
which completes the proof. 
\end{proof}


\section{depth comparison between ordinary and symbolic powers of $I(G)$}\label{sec4}

In this section, we discuss the following conjecture posed by the first-named author of this paper using local structure of graphs. We give a sufficient condition involving a leaf whose neighbor is weakly simplicial vertex. Our sufficient condition reduces the depth comparison to certain smaller induced subgraphs. As an application, we prove the conjecture for graphs whose connected components are unicyclic or bipartite. 

\begin{conj}[\cite{m2026}, Conjecture 1.6]\label{Conj by Muta}
Let $G$ be a graph. Then we have 
$$\depth S/I(G)^{(k)}\geq\depth S/I(G)^{k}\mbox{ for all }k\geq1.$$
\end{conj}

We begin with two colon identities for ordinary and symbolic powers at a weakly simplicial vertex. Here, a vertex $x$ of a graph $G$ is called {\it weakly simplicial }if the open neighborhood $N(x)$ of $x$ is an independent set and, for any $u,v\in N(x)$, either $N(u)\subset N(v)$ or $N(v)\subset N(u)$.

\begin{lemma}\label{colon-lemma}
Let $G$ be a graph. Suppose that $x$ is a weakly simplicial vertex. Let $u\in N(x)$ such that $N(u)\subset N(v)$ for all $v\in N(x)$. Then, for every $k\geq1$, we have 
\begin{enumerate}
\item $(I(G)^{k}:xu)=I(G)^{k-1}$.
\item $(I(G)^{(k)}:xu)=I(G)^{(k-1)}$, 
\end{enumerate}
where we set $I(G)^{0}=I(G)^{(0)}=S$. 
\end{lemma}
\begin{proof}
(1) Since $xu\in I(G)$, we have $(I(G)^{k}:xu)\supset I(G)^{k-1}$. Fix a monomial $m\in(I(G)^{k}:xu)$. Then, since $xum\in I(G)^{k}$, there exist edges $\{a_{i},b_{i}\}\in E(G)$ for $1\leq i\leq k$ such that $\prod_{i}a_{i}b_{i}$ divides $xum$. We distinguish the following cases:

{\bf Case 1.} Suppose that there exists $1\leq s\leq k$ such that $a_{s}b_{s}=xu$. Then, we see that $\prod_{i\neq s}a_{i}b_{i}$ divides $m$, and hence, we obtain $m\in I(G)^{k-1}$. 

{\bf Case 2.} Suppose that $a_{i}b_{i}\neq xu$ for every $i$ and $xu$ divides $\prod_{i}a_{i}b_{i}$. Then there exist $1\leq s,t\leq k$ with $s\neq t$ and $v, w\in V(G)$ such that $a_{s}b_{s}=xv$ and $a_{t}b_{t}=uw$. Since $v\in N(x)$ and the assumption, we know that $w\in N(u)\subset N(v)$, which implies that $\{v,w\}\in E(G)$. Hence, we see that $vw\prod_{i\neq s,t}a_{i}b_{i}=(\prod_{i}a_{i}b_{i})/xu$ divides $m$, and thus, $m\in I(G)^{k-1}$. 

{\bf Case 3.} Suppose that $a_{i}b_{i}\neq xu$ for every $i$ and $xu$ does not divide $\prod_{i}a_{i}b_{i}$. Take $z\in\{x,u\}$. If $z$ divides $\prod_{i}a_{i}b_{i}$, then we see that $\prod_{i\neq s}a_{i}b_{i}$ divides $(\prod_{i}a_{i}b_{i})/z$ for some $s$ and $(\prod_{i}a_{i}b_{i})/z$ divides $m$, and hence, we obtain $m\in I(G)^{k-1}$. Hence, we assume that $x$ and $u$ do not divide $\prod_{i}a_{i}b_{i}$. In this case, by the fact that $\prod_{i}a_{i}b_{i}$ divides $xum$, we know that $\prod_{i\neq s}a_{i}b_{i}$ divides $m$ for any $s$, which implies that $m\in I(G)^{k-1}$. 

(2) Let $C$ be a minimal vertex cover of $G$. Since $\{x,u\}\in E(G)$, we know $|C\cap\{x,u\}|\geq1$. Suppose that $x,u\in C$. Since $C\setminus\{x\}$ and $C\setminus\{u\}$ are not a vertex cover, there exist vertices $v\in N(x)\setminus C$ and $w\in N(u)\setminus C$. By assumption, we know $w\in N(u)\subset N(v)$, this means that $\{v,w\}\in E(G)$, a contradiction. Hence, we obtain $|C\cap\{x,u\}|=1$. For any monomial $m$, we know that $xum\in P_{C}^{k}$ if and only if $\sum_{v\in C}\deg_{v}(m)+1\geq k$. The latter one is equivalent to the condition that $m\in P_{C}^{k-1}$. Let $\Min(G)$ be the set of minimal vertex covers of $G$. Then, we obtain $(P_{C}^{k}:xu)=P_{C}^{k-1}$, and thus, we get 
$$(I(G)^{(k)}:xu)=\bigcap_{C\in\Min(G)}(P_{C}^{k}:xu)=\bigcap_{C\in\Min(G)}P_{C}^{k-1}=I(G)^{(k-1)},$$
which completes the proof. 
\end{proof}

To obtain a depth formula for ordinary powers, we use Eliahou--Kervaire splittings. We first recall the definition introduced in \cite{ek1990}. Let $I,J$, and $L$ be monomial ideals of $S$ such that $\mathcal{G}(I)=\mathcal{G}(J)\sqcup\mathcal{G}(L)$. Then, $I=J+L$ is an {\it Eliahou--Kervaire splitting} if there exists a splitting function 
$$\mathcal{G}(J\cap L)\rightarrow\mathcal{G}(J)\times\mathcal{G}(L)$$
sending $w\mapsto(\phi(w),\psi(w))$ such that 
\begin{enumerate}
\item $w=\lcm(\phi(w),\psi(w))$ for all $w\in\mathcal{G}(J\cap L)$, and 
\item for every subset $\emptyset\neq W\subset\mathcal{G}(J\cap L)$, $\lcm(\phi(W))$ and $\lcm(\psi(W))$ strictly divide $\lcm(W)$. 
\end{enumerate}

An Eliahou--Kervaire splitting gives the following formula for projective dimension. 

\begin{prop}[\cite{fhv2009}, Corollary 2.2]\label{pd formula}
Let $I,J$, and $L$ be monomial ideals of $S$. If $I=J+L$ is an Eliahou--Kervaire splitting, then we have $\pd I=\max\{\pd J, \pd L, \pd(J\cap L)+1\}$.
\end{prop}

Under the assumptions of Lemma \ref{colon-lemma}, we decompose $(I(G)^{k}:x)$ by deleting the vertex $u$. We prove that the decomposition is an Eliahou--Kervaire splitting. 

\begin{prop}\label{EK-splitting}
Let $G$ be a graph and let $x$ be a weakly simplicial vertex. Let $u\in N(x)$ such that $N(u)\subset N(v)$ for all $v\in N(x)$. Then, for every $k\geq2$, we have 
$$(I(G)^{k}:x)=(I(G-u)^{k}:x)+uI(G)^{k-1}\mbox{ and }(I(G-u)^{k}:x)\cap uI(G)^{k-1}=u(I(G-u)^{k}:x).$$
Moreover, $(I(G)^{k}:x)=(I(G-u)^{k}:x)+uI(G)^{k-1}$ is an Eliahou--Kervaire splitting if $(I(G-u)^{k}:x)\neq0$. 
\end{prop}
\begin{proof}
Fix $k\geq2$. First, we prove that $(I(G)^{k}:x)=(I(G-u)^{k}:x)+uI(G)^{k-1}$. Notice that we know that $(I(G-u)^{k}:x)\subset(I(G)^{k}:x)$. Also, since $xuI(G)^{k-1}\subset I(G)^{k}$, we get $uI(G)^{k-1}\subset(I(G)^{k}:x)$. Fix a monomial $m\in(I(G)^{k}:x)$. If $u$ divides $m$, then we have $xu(m/u)=xm\in I(G)^{k}$, which implies that $m/u\in I(G)^{k-1}$ from Lemma \ref{colon-lemma} (1). On the other hand, if $u$ does not divide $m$, then, since $xm\in K[V(G)\setminus\{u\}]$, we have $xm\in I(G)^{k}\cap K[V(G)\setminus\{u\}]=I(G-u)^{k}$, which implies that $m\in(I(G-u)^{k}:x)$. 
Hence, we get $(I(G)^{k}:x)=(I(G-u)^{k}:x)+uI(G)^{k-1}$. Moreover, since $(I(G-u)^{k}:x)\subset I(G)^{k-1}$, we obtain that 
\begin{align*}
(I(G-u)^{k}:x)\cap uI(G)^{k-1}&=(\lcm(a,ub)\,\,:a\in\mathcal{G}((I(G-u)^{k}:x)),b\in\mathcal{G}(I(G)^{k-1})) \\
&=u(\lcm(a,b)\,\,:a\in\mathcal{G}((I(G-u)^{k}:x)),b\in\mathcal{G}(I(G)^{k-1})) \\
&=u((I(G-u)^{k}:x)\cap I(G)^{k-1}) \\
&=u(I(G-u)^{k}:x).
\end{align*}
From now on, we prove that $(I(G)^{k}:x)=(I(G-u)^{k}:x)+uI(G)^{k-1}$ is an Eliahou--Kervaire splitting. Assume that $(I(G-u)^{k}:x)\neq0$. For any $m\in\mathcal{G}(I(G)^{k}:x)$, $xm\in I(G)^{k}$ implies that $\deg m\geq 2k-1$. On the other hand, every element of $\mathcal{G}(I(G)^{k-1})$ has degree $2k-2$. Let $m\in\mathcal{G}(I(G-u)^{k}:x)$ and $m^{\prime}\in\mathcal{G}(I(G)^{k-1})$. Since $xm\in I(G-u)^{k}$, we have $\deg m\geq 2k-1$, whereas $\deg m^{\prime}=2k-2$. Moreover, $u$ does not divide $m$. Hence, $um^{\prime}$ does not divide $m$. If $m$ divides $um^{\prime}$, then $m$ divides $m^{\prime}$, a contradiction. Hence, we get 
$$\mathcal{G}(I(G)^{k}:x)=\mathcal{G}(I(G-u)^{k}:x)\sqcup u\mathcal{G}(I(G)^{k-1}).$$ 
It remains to construct a splitting function. Write $N_{G-u}(x)=\{w_{1},\ldots, w_{t}\}$ and assume $N_{G-u}(w_{1})\subset\cdots\subset N_{G-u}(w_{t})$. Fix a total order on the variables of $K[V(G)\setminus\{u\}]$ in which $w_{1},\ldots, w_{t}$ occur first, followed by the remaining variables other than $x$, and $x$ occurs last. For every $m\in\mathcal{G}(I(G-u)^{k}:x)$, let $z(m)$ denote the first variable in this order which divides $m$. We claim that $m/z(m)\in I(G)^{k-1}$ for all $m\in\mathcal{G}(I(G-u)^{k}:x)$. Suppose that $z(m)=w_{i}$ for some $i$. Since $w_{1},\ldots, w_{i-1}$ do not divide $m$, we obtain $xm\in I(G-\{u,w_{1},\ldots, w_{i-1}\})^{k}$. Hence, from Lemma \ref{colon-lemma} (1), we see that $$m/w_{i}\in(I(G-\{u,w_{1},\ldots, w_{i-1}\})^{k}:xw_{i})\subset I(G-\{u,w_{1},\ldots, w_{i-1}\})^{k-1}\subset I(G)^{k-1}.$$

Suppose that no variables in $N_{G-u}(x)$ divides $m$. Take edges $\{a_{1},b_{1}\},\ldots,\{a_{k},b_{k}\}$ of a graph $G-u$ such that $\prod_{i}a_{i}b_{i}$ divides $xm$. Then we see that $\prod_{i}a_{i}b_{i}$ divides $m$. Since $\prod_{i}a_{i}b_{i}\in I(G-u)^{k}\subset(I(G-u)^{k}:x)$ and $m\in\mathcal{G}(I(G-u)^{k}:x)$, we have $m=\prod_{i}a_{i}b_{i}$. Take $s\in[k]$ such that $z(m)$ divides $a_{s}b_{s}$. Then we see that 
$$\frac{m}{z(m)}=\frac{a_{s}b_{s}}{z(m)}\prod_{i\neq s}a_{i}b_{i}\in I(G)^{k-1}.$$
For each $m\in\mathcal{G}(I(G-u)^{k}:x)$, we can take $m^{\prime}\in\mathcal{G}(I(G)^{k-1})$ such that $m^{\prime}$ divides $m/z(m)$. Define a map 
$$\mathcal{G}(u(I(G-u)^{k}:x))\rightarrow\mathcal{G}(I(G-u)^{k}:x)\times\mathcal{G}(uI(G)^{k-1}), \quad um\rightarrow(m,um^{\prime}).$$
Finally, we verify that this map satisfies the conditions given in an Eliahou--Kervaire splitting. 
Since $m^{\prime}$ divides $m$ and $u$ does not divide $m$, we have $\lcm(m,um^{\prime})=um$. Let $\emptyset\neq W\subset\mathcal{G}(u(I(G-u)^{k}:x))$. Since, $\phi(um)=m$ for every $um\in W$, we have $\lcm(W)=u\lcm(\phi(W))$, which implies that $\lcm(\phi(W))$ strictly divides $\lcm(W)$. Let $z$ be the first variable in the fixed order dividing at least one $m$ with $um\in W$. If $z$ divides $m$, then we have $z(m)=z$. Since $m^{\prime}$ divides $m/z(m)$, we have $\deg_{z}(m^{\prime})\leq\deg_{z}(m)-1$. If $z$ does not divide $m$, then we see that $z$ does not divide $m^{\prime}$. Thus, we obtain $\deg_{z}\lcm(\psi(W))<\deg_{z}\lcm(W)$. Moreover, we see that $um^{\prime}$ divides $um$ for each $um\in W$, which implies that $\lcm(\psi(W))$ divides $\lcm(W)$ and $\lcm(\psi(W))\neq\lcm(W)$, which completes the proof. 
\end{proof}

For symbolic powers, deleting a leaf gives the following lower bound for the depth. 

\begin{lemma}\label{leaf-depth}
Let $G$ be a graph and let $y$ be a leaf with $N(y)=\{x\}$. Then, for any $k\geq2$, we have 
\begin{align*}
\depth S/I(G)^{(k)}
\geq \min\left\{
\begin{aligned}
&\depth K[V(G)\setminus\{y\}]/I(G-y)^{(k)}+1,\\
&\depth S/I(G)^{(k-1)},\\
&\depth K[V(G)\setminus\{y\}]/(I(G-y)^{(k)}:x)
\end{aligned}
\right\}.
\end{align*}
\end{lemma}
\begin{proof}
Fix a minimal vertex cover $C$ of $G$. Since $y$ is a leaf, we know that $|C\cap\{x,y\}|=1$ and hence, we obtain $(I(G)^{(k)}:xy)=I(G)^{(k-1)}$. Hence, since $I(G-y)^{(k)}\subset I(G)^{(k)}$, we get $I(G-y)^{(k)}+xyI(G)^{(k-1)}\subset I(G)^{(k)}$. Also, fix a monomial $m\in I(G)^{(k)}$. If $y$ does not divide $m$, then we have $m\in I(G-y)^{(k)}$. Moreover, if $xy$ divides $m$, then we have $m\in xyI(G)^{(k-1)}$. Hence, suppose that $y$ divides $m$ and $x$ does not divide $m$. For any minimal vertex cover $D$ of $G-y$, since $D\cup\{x\}$ is a vertex cover of $G$, there exists a minimal vertex cover $C$ of $G$ such that $C\subset D\cup\{x\}$.  Thus, we obtain 
$$\sum_{v\in D}\deg_{v}(m)=\sum_{v\in D\cup\{x\}}\deg_{v}(m)\geq\sum_{v\in C}\deg_{v}(m)\geq k,$$
which implies that $m\in I(G-y)^{(k)}$. As a consequence, we get $$I(G)^{(k)}=I(G-y)^{(k)}+xyI(G)^{(k-1)}.$$
Moreover, since 
\begin{align*}
I(G-y)^{(k)}\cap xyI(G)^{(k-1)}&=(\lcm(a,xyb)\,\,: a\in\mathcal{G}(I(G-y)^{(k)}),b\in\mathcal{G}(I(G)^{(k-1)})) \\
&=xy(\lcm(a/\gcd(a,x), b)\,\,: a\in\mathcal{G}(I(G-y)^{(k)}),b\in\mathcal{G}(I(G)^{(k-1)})) \\
&=xy((I(G-y)^{(k)}:x)\cap I(G)^{(k-1)}) \\
&=xy(I(G-y)^{(k)}:x),
\end{align*}
we have the following short exact sequence:
$$0\rightarrow S/xy(I(G-y)^{(k)}:x)\rightarrow S/I(G-y)^{(k)}\oplus S/xyI(G)^{(k-1)}\rightarrow S/I(G)^{(k)}\rightarrow0.$$
Therefore, from the depth lemma, we obtain the desired inequality, as required. 
\end{proof}

We next compute the depth of $S/(I(G)^{k}:x)$ and give a lower bound for the corresponding symbolic quotient. 

\begin{lemma}\label{depth ineq}
Let $G$ be a graph and let $x$ be a weakly simplicial vertex. Let $u\in N(x)$ such that $N(u)\subset N(v)$ for all $v\in N(x)$. Then, for every $k\geq2$, we have 
$$\depth S/(I(G)^{k}:x)=\min\{\depth S/I(G)^{k-1},\depth K[V(G)\setminus\{u\}]/(I(G-u)^{k}:x)\}$$ and 
$$\depth S/(I(G)^{(k)}:x)\geq\min\{\depth S/I(G)^{(k-1)},\depth K[V(G)\setminus\{u\}]/(I(G-u)^{(k)}:x)\}$$
\end{lemma}
\begin{proof}
Fix $k\geq2$. If $(I(G-u)^{k}:x)=0$, then from Proposition \ref{EK-splitting}, we have $(I(G)^{k}:x)=uI(G)^{k-1}$, which implies that $\depth S/(I(G)^{k}:x)=\depth S/I(G)^{k-1}\leq |V(G)|-1$. If $(I(G-u)^{k}:x)\neq0$, then the formula follows from Propositions \ref{pd formula} and \ref{EK-splitting} and the Auslander--Buchsbaum formula. Hence, it is enough to consider symbolic powers. From Lemma \ref{colon-lemma} (2) and $I(G)^{(k)}\cap K[V(G)\setminus\{u\}]=I(G-u)^{(k)}$, we have the short exact sequence 
$$0\rightarrow S/I(G)^{(k-1)}(-1)\rightarrow S/(I(G)^{(k)}:x)\rightarrow K[V(G)\setminus\{u\}]/(I(G-u)^{(k)}:x)\rightarrow0.$$
The desired inequality follows from the depth lemma, as required. 
\end{proof}

We now combine these results to give a sufficient condition for the depth comparison. The condition uses a leaf with a weakly simplicial neighbor and reduces the comparison to smaller induced subgraphs.

\begin{thm}\label{comparison}
Let $G$ be a graph and let $y$ be a leaf of $G$ with $N(y)=\{x\}$. Suppose that $x$ is weakly simplicial, $N(x)\setminus\{y\}=\{u_{1},\ldots, u_{r}\}$ with $r\geq1$, and $N(u_{1})\subset\cdots\subset N(u_{r})$. Suppose that the following conditions hold:
\begin{enumerate}
\item $\depth S_{i}/I(G-\{y,u_{1},\ldots, u_{i}\})^{(k)}\geq\depth S_{i}/I(G-\{y,u_{1},\ldots, u_{i}\})^{k}$ for all $0\leq i<r$ and $k\geq1$,
where $S_{i}=K[V(G)\setminus\{y,u_{1},\ldots, u_{i}\}]$ be a polynomial ring over $K$, with $S_{0}=K[V(G)\setminus\{y\}]$. 
\item $\depth S/I(G-N[x])^{(k)}\geq\depth S/I(G-N[x])^{k}$ for all $k\geq1$. 
\end{enumerate}
Then we have $\depth S/I(G)^{(k)}\geq\depth S/I(G)^{k}$ for all $k\geq1$. 
\end{thm}
\begin{proof}
We use induction on $k$. The assertion is clear for $k=1$. Hence, we assume that $k\geq2$. Notice that $x$ is weakly simplicial in $G-\{y,u_{1},\ldots, u_{i}\}$ for any $0\leq i<r$. From Lemma \ref{depth ineq}, we have 
\begin{align*}
\depth S_{0}/(I(G-y)^{k}:x)
= \min\left\{
\begin{aligned}
&\min_{0\leq i<r}\{\depth S_{i}/I(G-\{y,u_{1},\ldots, u_{i}\})^{k-1}\},\\
&\depth K[V(G)\setminus N[x]]/I(G-N[x])^{k}+1.
\end{aligned}
\right\}.
\end{align*}
and 
\begin{align*}
\depth S_{0}/(I(G-y)^{(k)}:x)
\geq \min\left\{
\begin{aligned}
&\min_{0\leq i<r}\{\depth S_{i}/I(G-\{y,u_{1},\ldots, u_{i}\})^{(k-1)}\},\\
&\depth K[V(G)\setminus N[x]]/I(G-N[x])^{(k)}+1.
\end{aligned}
\right\}.
\end{align*}
Hence, from assumptions (1) and (2), we have $$\depth S_{0}/(I(G-y)^{(k)}:x)\geq\depth S_{0}/(I(G-y)^{k}:x).$$
Now, from \cite[Proposition 1.2 and Corollary 1.3]{r2010}, we obtain that 
$$\depth S/I(G)^{k}\leq\depth S_{0}/I(G-y)^{k}+1\mbox{ and }\depth S/I(G)^{k}\leq\depth S/(I(G)^{k}:x).$$
Since $N(y)=\{x\}\subset N(v)$ for any $v\in N(x)$, from Lemma \ref{depth ineq}, we see that 
$$\depth S/I(G)^{k}\leq\depth S/(I(G)^{k}:x)=\min\{\depth S/I(G)^{k-1},\depth S_{0}/(I(G-y)^{k}:x)\},$$
and thus, we get 
$$\depth S/I(G)^{k}\leq\min\{\depth S_{0}/I(G-y)^{k}+1,\depth S/I(G)^{k-1},\depth S_{0}/(I(G-y)^{k}:x)\}.$$
Therefore, from Lemma \ref{leaf-depth}, assumption (1), and the induction hypothesis, we conclude that 
\begin{align*}
\depth S/I(G)^{(k)}&\geq\min\{\depth S_{0}/I(G-y)^{(k)}+1,\depth S/I(G)^{(k-1)},\depth S_{0}/(I(G-y)^{(k)}:x)\} \\
&\geq\min\{\depth S_{0}/I(G-y)^{k}+1,\depth S/I(G)^{k-1},\depth S_{0}/(I(G-y)^{k}:x)\}\\
&\geq\depth S/I(G)^{k},
\end{align*}
which completes the proof. 
\end{proof}

Before applying Theorem \ref{comparison}, we show that the depth comparison is preserved under disjoint unions.  

\begin{lemma}\label{disconnected}
Let $G=G_{1}\sqcup G_{2}$ be a graph and set $S_{i}=K[V(G_{i})]$ for $i=1,2$. Suppose that $\depth S_{i}/I(G_{i})^{(k)}\geq\depth S_{i}/I(G_{i})^{k}$ for $i=1,2$ and all $k\geq1$. Then we have 
$$\depth S/I(G)^{(k)}\geq\depth S/I(G)^{k}\mbox{ for all }k\geq2.$$
\end{lemma}
\begin{proof}
If $I(G_{1})=0$, then, for every $k\geq1$, we have 
$$\depth S/I(G)^{(k)}=|V(G_{1})|+\depth S_{2}/I(G_{2})^{(k)}$$
and $$\depth S/I(G)^{k}=|V(G_{1})|+\depth S_{2}/I(G_{2})^{k},$$
which imply that the assertion follows from the assumption on $G_{2}$. Hence, we may assume that both $G_{1}$ and $G_{2}$ have at least one edge. From \cite[Theorem 1.1]{nv2019} and \cite[Theorems 5.6 and 5.11]{hntt2020}, we know that 
\begin{align*}
\depth S/I(G)^{k}
= \min_{i\in[1,k-1],j\in[1,k]}\left\{
\begin{aligned}
&\depth S_{1}/I(G_{1})^{k-i}+\depth S_{2}/I(G_{2})^{i}+1,\\
&\depth S_{1}/I(G_{1})^{k-j+1}+\depth S_{2}/I(G_{2})^{j}.
\end{aligned}
\right\}
\end{align*}
and 
\begin{align*}
\depth S/I(G)^{(k)}
= \min_{i\in[1,k-1],j\in[1,k]}\left\{
\begin{aligned}
&\depth S_{1}/I(G_{1})^{(k-i)}+\depth S_{2}/I(G_{2})^{(i)}+1,\\
&\depth S_{1}/I(G_{1})^{(k-j+1)}+\depth S_{2}/I(G_{2})^{(j)}.
\end{aligned}
\right\}.
\end{align*}
Hence, the assertion follows. 
\end{proof}

We conclude this paper with the following theorem, which proves Conjecture \ref{Conj by Muta} for graphs whose connected components are unicyclic or bipartite. 

\begin{thm}\label{unicyclic}
Let $G$ be a graph. Suppose that every connected component of $G$ is unicyclic or bipartite. Then we have 
$$\depth S/I(G)^{(k)}\geq\depth S/I(G)^{k}\mbox{ for all }k\geq1.$$
\end{thm}
\begin{proof}
We use strong induction on $n=|V(G)|$. If $G$ is bipartite, then, from \cite[Theorem 5.9]{svv1994}, we know that $I(G)^{(k)}=I(G)^{k}$ for all $k\geq1$. In particular, the assertion follows when $n\leq2$. Hence, from Lemma \ref{disconnected}, we may assume that $G$ is connected and non-bipartite. By assumption, $G$ has exactly one cycle $C$, and $C$ is odd. Write $|V(C)|=2t+1$.  
For $v,w\in V(G)$, let $\dist_{G}(v,w)$ be the minimum number of edges in a path joining $v$ and $w$. Set $\dist_{G}(v,V(C))=\min\{\dist_{G}(v,c)\,\,: c\in V(C)\}$. 

First, suppose $\max_{v\in V(G)}\dist_{G}(v,V(C))\leq1$. For $1\leq k\leq t$, from \cite[Lemma 3.10]{rty2011}, we know $I(G)^{(k)}=I(G)^{k}$. Let $k\geq t+1$, and take an edge $\{a,b\}\in E(C)$. Set 
$m=(ab)^{k-t-1}\prod_{v\in V(C)}v$. Since $\deg m=2k-1$, we see $m\notin I(G)^{k}$. Fix $w\in V(G)$. Take $c\in V(C)$ such that $\{w,c\}\in E(G)$. The induced subgraph $C|_{V(C)\setminus\{c\}}$ is a path with $2t$ vertices, and hence, there exist edges $\{a_{1},b_{1}\},\ldots,\{a_{t},b_{t}\}\in E(C)$ such that $\prod_{i}a_{i}b_{i}=\prod_{v\in V(C)\setminus\{c\}}v$, which implies that 
$$wm=(wc)\left(\prod_{i}a_{i}b_{i}\right)(ab)^{k-t-1}\in I(G)^{k}.$$
Therefore, we see that $\depth S/I(G)^{k}=0$, and thus, we obtain the desired inequality. 

Next, suppose that $d=\max\{\dist_{G}(v,V(C))\,\,: v\in V(G)\}\geq2$. Since $G$ is connected and $C$ is its unique cycle, deleting all edges of $C$ gives a forest whose connected components each contain exactly one vertex of $C$. By uniqueness of paths in a tree, every $v\notin V(C)$ has exactly one neighbor $w$ satisfying $\dist_{G}(w,V(C))=\dist_{G}(v,V(C))-1$. Moreover, every other neighbor $z\in N(v)\setminus\{w\}$ satisfies $\dist_{G}(z,V(C))=\dist_{G}(v,V(C))+1$. Take $y\in V(G)$ such that $\dist_{G}(y,V(C))=d$, and let $x\in N(y)$ with $\dist_{G}(x,V(C))=d-1$. Since $d\geq2$, we have $x\notin V(C)$. Set $r=|N(x)|-1\geq1$ and denote by $u_{r}$ the unique neighbor of $x$ with $\dist_{G}(u_{r},V(C))=d-2$. Notice that $u_{r}\neq y$. For every $z\in N(x)\setminus\{u_{r}\}$, we have $\dist_{G}(z,V(C))=d$. If $w\in N(z)\setminus\{x\}$, we have $\dist_{G}(w,V(C))=d+1$, a contradiction. Hence, we obtain that $N(z)=\{x\}$ for every $z\in N(x)\setminus\{u_{r}\}$. Write $N(x)\setminus\{y,u_{r}\}=\{u_{1},\ldots, u_{r-1}\}$. Then, we see that $N(y)=\{x\}$, $N(u_{i})=\{x\}\subset N(u_{r})$ for any $1\leq i<r$, and $N(x)$ is an independent set of $G$. Hence, $x$ is weakly simplicial and $N(u_{1})\subset\cdots\subset N(u_{r})$. Each connected component of $G-\{y,u_{1},\ldots, u_{i}\}$ and $G-N[x]$ contains at most one cycle and is therefore bipartite or unicyclic for $0\leq i<r$. Hence, by the induction hypothesis, we see that assumptions hold in Theorem \ref{comparison}, the assertion holds, as required. 
\end{proof}


\section*{Acknowledgement}
The research of Yuji Muta was partially supported by ohmoto-ikueikai and JST SPRING Japan Grant Number JPMJSP2126. The research of Naoki Terai was partially supported by JSPS Grant-in Aid for Scientific Research (C) 24K06670.



\begin{thebibliography}{99}

\bibitem{BHBook} W. Bruns and J. Herzog,
{\it Cohen--Macaulay Rings},
Cambridge Stud. Adv. Math., vol. 39, Cambridge Univ. Press, Cambridge, 1993.

\bibitem{ek1990} S. Eliahou and M. Kervaire, 
Minimal resolutions of some monomial ideals, 
{\it J. Algebra} {\bf 129} (1990), no. 1, 1--25.

\bibitem{f1982}
O. Favaron,
Very well covered graphs,
{\it Discrete Math.} {\bf 42} (1982), no. 2--3, 177--187.

\bibitem{fmm2026Rees} A. Ficarra, S. Moradi, and Y. Muta, 
Symbolic Rees algebras of complementary edge ideals, 
arXiv:2609.01165 (2026).

\bibitem{fmm2026reg-depth} A. Ficarra, S. Moradi, and Y. Muta, 
Regularity, depth and Serre depth of symbolic powers of complementary edge ideals, 
arXiv:2610.23224 (2026).

\bibitem{fhv2009} C.A. Francisco, H.T. Ha and A. Van Tuyl, 
Splittings of monomial ideals, 
{\it Proc. Amer. Math. Soc.} {\bf 137} (2009), no. 10, 3271--3282.

\bibitem{gv2005} I. Gitler and C.E. Valencia,
Bounds for invariants of edge-rings,
{\it Comm. Algebra} \textbf{33} (2005), 1603--1616.

\bibitem{hntt2020} H.T. Ha, H.D. Nguyen, N.V. Trung and T.N. Trung, 
Symbolic powers of sums of ideals, 
{\it Math. Z.} {\bf 294} (2020), 1499--1520.

\bibitem{HHBook} J. Herzog and T. Hibi, 
\emph{Monomial ideals}, 
Graduate texts in Mathematics {\bf 260}, Springer, 2011.

\bibitem{hlt2024} H.T.T. Hien, H.M. Lam, and N.V. Trung,
Decreasing behavior of the depth functions of edge ideals,
{\it J. Algebraic Combin.} {\bf 59} (2024), no. 1, 37--53.

\bibitem{hktt2017} L.T. Hoa, K. Kimura, N. Terai and T.N. Trung, 
Stability of depths of symbolic powers of Stanley--Reisner ideals, 
{\it J. Algebra} {\bf 473} (2017), 307--323. 

\bibitem{ht2010} L.T. Hoa and T.N. Trung,
Partial Castelnuovo--Mumford regularities of sums and intersections of powers of monomial ideals,
{\it Math. Proc. Cambridge Philos. Soc.} {\bf 149} (2010), no. 2, 229--246. 

\bibitem{kty2018}
K. Kimura, N. Terai and S. Yassemi,
The projective dimension of the edge ideal of a very well-covered graph,
{\it Nagoya Math. J.} {\bf 230} (2018), 160--179.

\bibitem{mtv2024} N.C. Minh, T.N. Trung and T. Vu, 
Stable value of depth of symbolic powers of edge ideals of graphs, 
{\it Pacific J. Math.} {\bf 329} (2024), 147--164. 

\bibitem {mv2012} S. Morey and R.H. Villarreal, 
{\it Edge ideals: algebraic and combinatorial properties}, Progress in commutative algebra 1, 85--126. Walter de Gruyter \& Co., Berlin, 2012


\bibitem{m2026} Y. Muta, 
Comparisons between ordinary and symbolic powers of edge ideals with respect to regularity and depth, 
arXiv:2609.05945 (2026).

\bibitem{mt2025} Y. Muta and N. Terai, 
The Serre depth of Stanley--Reisner rings and the depth of their symbolic powers,
arXiv: 2509.14838. (2025).

\bibitem{mt2026} Y. Muta and N. Terai, 
Algebraic study on rooted products of graphs and multi-clique corona graphs, 
{\it J. Pure Appl. Algebra} {\bf 230} (2026), no. 8. 

\bibitem{nt2019} H.D. Nguyen and N.V. Trung, 
Depth functions of symbolic powers of homogeneous ideals, 
{\it Invent. math.} {\bf 218} (2019), 779--827. 

\bibitem{nv2019} H.D. Nguyen and T. Vu,
Powers of sums and their homological invariants,
{\it J. Pure Appl. Algebra} {\bf 223} (2019), no. 7, 3081--3111.

\bibitem{r2010} A. Rauf, 
Depth and Stanley depth of multigraded modules,
{\it Comm. Algebra} {\bf 38} (2010), 773--784.

\bibitem{rty2011} G. Rinaldo, N. Terai and K. Yoshida,
Cohen--Macaulayness for symbolic power ideals of edge ideals,
{\it J. Algebra} {\bf 347} (2011), 1--22.

\bibitem{sf2017}
S.A. Seyed Fakhari,
Stanley depth and symbolic powers of monomial ideals,
{\it Math. Scand.} {\bf 120} (2017), no. 1, 5--16.

\bibitem{svv1994} A. Simis, W.V. Vasconcelos and R.H. Villarreal,
On the ideal theory of graphs, 
{\it J. Algebra} {\bf 167} (1994), 389--416.

\bibitem{t2005} Y. Takayama,
Combinatorial characterizations of generalized Cohen--Macaulay monomial ideals,
{\it Bull. Math. Soc. Sci. Math. Roumanie (N.S.)}
{\bf 48(96)} (2005), no. 3, 327--344.

\bibitem{v2011} M. Varbaro, 
Symbolic powers and matroids, 
{\it Proc. Amer. Math. Soc.} {\bf 139} (2011), 2357--2366. 

\bibitem{v1990} R.H. Villarreal, 
Cohen--Macaulay graphs, 
{\it Manuscripta Math.} {\bf 66} (1990), no. 3, 277--293. 

\bibitem{VBook} R. H. Villarreal,
\emph{Monomial Algebras}, 3rd edition,
CRC Press, Boca Raton, FL, 2026.

\end{thebibliography}
\end{document}